\documentclass[11pt, reqno]{amsart}

\usepackage{mathtools, amsthm, amssymb}
\mathtoolsset{mathic}
\usepackage{microtype}
\usepackage{tikz}
\usetikzlibrary{calc,arrows.meta,decorations.pathreplacing,positioning}

\usepackage[dvipsnames]{xcolor}
\usepackage{todonotes}
\usepackage[
colorlinks=true,
linkcolor=RoyalBlue,
citecolor=BrickRed,
urlcolor=ForestGreen
]{hyperref}

\newcommand{\R}{\mathbf{R}}

\newcommand{\ball}{\mathbf{B}}
\newcommand{\sphere}{\mathbf{S}}

\DeclareMathOperator{\E}{\mathbf{E}}
\DeclareMathOperator{\prob}{\mathbf{P}}
\DeclareMathOperator{\area}{area}
\DeclareMathOperator{\interior}{int}
\DeclareMathOperator{\vol}{vol}
\newcommand{\voln}{\vol_n}

\DeclarePairedDelimiter\abs{\lvert}{\rvert}
\DeclarePairedDelimiter\norm{\lVert}{\rVert}

\newcommand{\blank}{{\mspace{1mu}\cdot\mspace{1mu}}}

\newtheorem{theorem}{Theorem}[section]

\newtheorem{proposition}[theorem]{Proposition}

\newtheorem{conjecture}[theorem]{Conjecture}

\theoremstyle{definition}

\theoremstyle{remark}

\numberwithin{equation}{section}

\renewcommand{\geq}{\geqslant}
\renewcommand{\le}{\leqslant}
\renewcommand{\ge}{\geqslant}
\renewcommand{\vec}{\mathbf}
\newcommand{\diff}{\mathop{}\!d}

\begin{document}

\title{Straight-line optimality in Bellman's lost-in-a-forest problem for Euclidean balls}

\author{David Treeby}
\address{School of Mathematics, Monash University, Australia}
\email{david.treeby@monash.edu.au}

\author{Edward Wang}
\address{School of Mathematics and Statistics, University of Melbourne, Australia}
\email{edward.wang@student.unimelb.edu.au}


\begin{abstract}
  We prove that among all unit-speed paths, a straight line minimises the expected escape time from a ball in $\mathbf{R}^n$, solving the min-mean variant of Bellman's Lost~in~a~Forest problem for ball-shaped forests. The proof uses the Kneser--Poulsen conjecture in the plane, together with results on polygonal chain straightening in higher dimensions. Moreover, we calculate this minimal escape time by deriving the expected linear distance to the boundary of a ball in $n$ dimensions.
\end{abstract}

\date{\today}

\maketitle

\section{Introduction}

In 1955, Richard Bellman~\cite{bellman} posed the now famous ``Lost~in~a~Forest'' problem, which was later popularised by Finch and Wetzel in \cite{finch}. The problem can be phrased as follows:
\begin{quote}
  A hiker is lost in a forest whose shape and dimensions are precisely known to him. What is the best path for him to follow to escape from the forest?
\end{quote}
Bellman's original note proposed two interpretations of ``best''. In the first, the \emph{maximum time} to escape is minimised; in the second, the \emph{expected time} is minimised (both assuming a unit-speed path). The first interpretation has been studied extensively, with a recent preprint by Deng~\cite{Deng24} proposing a general solution for forest boundaries of any shape. The second interpretation, however, has seen surprisingly little attention. To the best of our knowledge, it has been explored in the literature on only two occasions: in a 1961 paper by Gluss~\cite{Gluss61}, and a more recent work by Finch and Shonder~\cite{Finch16}, who call for more study on this variation. Indeed, Finch and Wetzel's survey~\cite{finch} notes that ``little is known'' about this interpretation, whilst Ward~\cite{Ward08} mentions that it ``generally [has] not been investigated''. We speculate that the little attention given to this variation is borne of its difficulty. To date, no definitive and unconditional result has been found for \emph{any} domain. 

In this paper, we take the forest to be the unit disk in $\R^2$, later extending to unit balls in $\R^n$. The maximal time (min-max) problem for such forests was solved long ago---in fact, Finch and Wetzel~\cite{finch} note that this combination of forest shape and optimisation criterion was the first variation of Bellman's problem to be investigated at all---by Gross~\cite{Gross55} in 1955. We therefore focus on the expected time (min-mean) problem, where the hiker's initial position is assumed to be uniformly distributed in the forest. Finch and Shonder~\cite{Finch16} were the first to consider this variation, conjecturing that the optimal path is a straight line, but proving this only for paths of small curvature. Here, we prove its optimality without any restrictions, using the Kneser--Poulsen conjecture, proven in $\R^2$ by Bezdek and Connelly in \cite{BezdekConnelly2002}, as our main tool. In doing so, we hope to spark interest in this neglected variation of the Lost~in~a~Forest problem. We later generalise this result to higher dimension analogues, proving that a straight line path remains optimal for $n$-balls.

We note that the optimality of the straight line is not obvious a priori. Since the mass of the uniform distribution concentrates near the boundary of a ball, a hiker is more likely to start near the edge. One might therefore suspect that a local search strategy, attempting to detect the boundary quickly before committing to a trajectory, could yield a lower expected time than a blind linear path. Our result demonstrates that the global geometry of the ball overrides this local intuition.

Finally, we derive the exact minimal expected escape time for $n$-balls when travelling along a straight line path, and conclude the paper with a discussion of open conjectures involving forests shaped as regular polygons.

\section{Lost in a disk}

We begin in two dimensions by examining a forest in the shape of the unit disk $\ball^2 = \{ x \in \R^2 : \norm{x} \le 1 \}$, with $\norm{\blank}$ being the Euclidean norm.

A \emph{path} is a continuous curve $\gamma \colon [0, \infty) \to \R^2$ that is arc-length parametrised and satisfies $\gamma(0) = 0$. That is, for any $t > 0$, the length of the curve from $\gamma(0)$ to $\gamma(t)$ is exactly $t$. Using such a definition will allow for curves with sharp corners, such as the polygonal chains which we require later. The \emph{escape time} $T_{\gamma}(x)$ of a path $\gamma$ starting from a point $x \in \ball^2$ is defined as \[
  T_{\gamma}(x) \coloneq \inf \{ t \ge 0 : \norm{x+\gamma(t)} \ge 1 \},
\] the first time at which the path reaches the boundary of the forest. If the path never exits the unit disk, we set $T_\gamma(x) = \infty$. Due to the rotational symmetry of $\ball^2$, the distribution of the escape time is invariant under rotations of the path. Hence, we may fix the orientation of a path without loss of generality.

Let $X$ be a random variable distributed uniformly in $\ball^2$. The \emph{expected escape time} of a path $\gamma$ is \[
  J(\gamma) \coloneq \E[T_{\gamma}(X)].
\] With this notation in place, we state our result in Theorem~\ref{thm:2d}.

\begin{theorem} \label{thm:2d}
  Let \(\gamma\) be any path in \(\R^2\). Then, the expected escape time from the unit disk \(\ball^2\) satisfies \[
    J(\gamma) \ge J(\ell),
  \] where \(\ell\) is a straight line.
\end{theorem}

We first establish a geometric representation of the expected escape time. As $T_{\gamma}(X)$ is a non-negative random variable, we may write its expected value as an integral of the tail probability,
\begin{equation} \label{eqn:tail_integral}
  J(\gamma) = \E[T_{\gamma}(X)] = \int_{0}^{\infty} \prob(T_{\gamma}(X) > t) \diff t.
\end{equation}

Define the \emph{strict non-escape region} at time $t$ as the set of all starting points $x$ that remain strictly within the forest up to time $t$: \[
  U_{\gamma}(t) \coloneq \{ x \in \ball^2 : \norm{x + \gamma(s)} < 1\ \text{for all}\ 0 \le s \le t \}.
  \] Clearly, $T_{\gamma}(x) > t$ if and only if $x \in U_{\gamma}(t)$. Let $\ball(x)$ denote the closed unit disk centered at $x$. The region $U_{\gamma}(t)$ can be equivalently written as the intersection of open disks, \[
  U_{\gamma}(t) = \bigcap_{s \in [0, t]} \interior\bigl(\ball(-\gamma(s))\bigr).
  \] To compute its area, we instead consider the intersection of the corresponding closed disks: \[
  S_{\gamma}(t) \coloneq \bigcap_{s \in [0, t]} \ball(-\gamma(s)).
\] As an arbitrary intersection of closed, convex sets, $S_{\gamma}(t)$ is itself a closed convex set. A standard result in geometric measure theory (see, e.g., Lang~\cite[Theorem 1]{Lang86}) establishes that the boundary of any convex set has Lebesgue measure zero. Consider a point $x \in S_{\gamma}(t) \setminus U_{\gamma}(t)$. Such an $x$ must satisfy $\norm{x + \gamma(s)} = 1$ for some $s \in [0, t]$, which prevents $x$ from lying in the interior of $S_{\gamma}(t)$, since any open ball around $x$ would inevitably contain points outside $\ball(-\gamma(s))$ and therefore outside $S_{\gamma}(t)$. Consequently, $S_{\gamma}(t) \setminus U_{\gamma}(t) \subseteq \partial S_{\gamma}(t)$, which is the boundary of a convex set and thus has measure zero. Therefore $U_{\gamma}(t)$ and $S_{\gamma}(t)$ differ only by a set of measure zero and hence have identical area.

The \emph{non-escape probability} $\prob(T_{\gamma}(X) > t)$ is thus equal to the ratio of the area of $S_{\gamma}(t)$ to the area of $\ball^2$, yielding \[
  \prob(T_{\gamma}(X) > t) = \frac{\area(S_{\gamma}(t))}{\area(\ball^2)} = \frac{1}{\pi} \area(S_{\gamma}(t)).
\] Substituting this back into \eqref{eqn:tail_integral} gives the key identity
\begin{equation} \label{eqn:area_integral}
  J(\gamma) = \frac{1}{\pi}\int_{0}^{\infty} \area(S_{\gamma}(t)) \diff t,
\end{equation}
creating a clear geometric objective---in order to minimise the expected escape time, we must minimise the integral of the area of $S_{\gamma}(t)$. We do this pointwise by bounding the integrand, which is an intersection of closed unit disks, below, using the Kneser--Poulsen conjecture. Specifically, we use the two-dimensional case, proved by Bezdek and Connelly in \cite{BezdekConnelly2002}.

If a finite collection of unit balls in $\R^n$ is rearranged such that the distance between each pair of centers does not decrease, we say that such a rearrangement is an \emph{expansion}. Formally, if $\vec{p} = (\vec{p}_1, \dots, \vec{p}_N)$ and $\vec{q} = (\vec{q}_1, \dots, \vec{q}_N)$ are two configurations of $N$ points in $\R^n$ such that $\norm{ \vec{p}_i - \vec{p}_j } \le \norm{ \vec{q}_i - \vec{q}_j }$ for all pairs $i$ and $j$, then $\vec{q}$ is an expansion of $\vec{p}$.

The conjecture, posed independently by Kneser~\cite{Kneser1955} and Poulsen~\cite{Poulsen1954}, states that expansions do not decrease the volume of the union of the balls, and do not increase the volume of the intersections. The conjecture remains wide open for arbitrary expansions beyond $\R^2$, despite the result seeming geometrically intuitive. Bezdek and Connelly's planar proof even handles disks of varying radii; we give the formal statement in Theorem~\ref{thm:KP2}. For a recent survey of the Kneser--Poulsen conjecture, we refer the reader to \cite{Csikos2018}.

\begin{theorem}[Kneser--Poulsen in the plane] \label{thm:KP2}
  Let \(\vec{q} = (\vec{q}_1, \dots, \vec{q}_N)\) be an expansion of \(\vec{p} = (\vec{p}_1, \dots, \vec{p}_N)\) in \(\R^2\), and let \(\ball(\vec{p}_i, r_i)\) denote the closed disk centered at \(\vec{p}_i\) with radius \(r_i\). Then,
  \begin{equation}
    \area \biggl(\, \bigcup_{i = 1}^{N} \ball(\vec{p}_i, r_i) \biggr) \le \area \biggl(\, \bigcup_{i = 1}^{N} \ball(\vec{q}_i, r_i) \biggr)
  \end{equation} and 
  \begin{equation} \label{eqn:KP2}
    \area \biggl(\, \bigcap_{i = 1}^{N} \ball(\vec{p}_i, r_i) \biggr) \ge \area \biggl(\, \bigcap_{i = 1}^{N} \ball(\vec{q}_i, r_i) \biggr)
  \end{equation}
  for any choice of radii \(r_1, \dots, r_N > 0\).
\end{theorem}

Armed with this tool, our proof of Theorem~\ref{thm:2d} hinges on the observation that, for any partition $P$, the sampled configuration along a straight line is an expansion of the sampled configuration along an arbitrary path.

\begin{figure}[htbp]
  \centering
  \begin{tikzpicture}[
    scale=1,
    >=Stealth,
    % Styles
    pathnode/.style={circle, fill=black!80, inner sep=2pt},
    curve/.style={line width=1.8pt, color=blue!65!black, line cap=round},
    chord/.style={line width=1pt, color=black!50, dashed},
    straight/.style={line width=1.8pt, color=red!70!black, line cap=round},
    arrowstyle/.style={-{Stealth[length=3mm]}, line width=1.8pt}
    ]

    % Curved path (top)
    \coordinate (P0) at (0, 2.5);
    \coordinate (P1) at (1.8, 3.3);
    \coordinate (P2) at (4.0, 2.1);
    \coordinate (P3) at (6.0, 3.0);
    \coordinate (PEnd) at (7.2, 3.5);

    % Draw curved path
    \draw[curve] (P0) to[out=45, in=170] (P1) 
      to[out=-10, in=160] (P2) 
      to[out=-20, in=240] (P3);
    \draw[arrowstyle, color=blue!65!black] (P3) to[out=70, in=180] (PEnd) 
      node[right, font=\normalsize, color=blue!65!black] {$\gamma$};

    % Chords
    \draw[chord] (P0) -- (P1);
    \draw[chord] (P1) -- (P2);
    \draw[chord] (P2) -- (P3);

    % Points on curve
    \foreach \i in {0,1,2,3} {
      \node[pathnode] at (P\i) {};
    }

    % Labels for curve
    \node[font=\small] at ($(P0)+(-0.25,-0.25)$) {$t_0$};
    \node[font=\small] at ($(P1)+(0,0.35)$) {$t_1$};
    \node[font=\small] at ($(P2)+(0,-0.35)$) {$t_2$};
    \node[font=\small] at ($(P3)+(-0.3,0.3)$) {$t_3$};

    % Straight line (bottom)
    \def\yLine{0}
    \def\s{2.5}
    \def\sTwo{2.9}
    \def\sThree{2.5}

    \coordinate (Q0) at (0, \yLine);
    \coordinate (Q1) at ($(\s, \yLine)$);
    \coordinate (Q2) at ($(Q1)+(\sTwo, 0)$);
    \coordinate (Q3) at ($(Q2)+(\sThree, 0)$);

    % Draw straight line
    \draw[straight] (Q0) -- (Q3);
    \draw[arrowstyle, color=red!70!black] (Q3) -- ++(1.2, 0)
      node[right, font=\normalsize, color=red!70!black] {$\ell$};

    % Points on line
    \foreach \i in {0,1,2,3} {
      \node[pathnode] at (Q\i) {};
    }

    % Labels for line
    \foreach \i in {0,1,2,3} {
      \node[font=\small] at ($(Q\i)+(0,-0.4)$) {$t_\i$};
    }

    % Add braces to show the expansion relationship
    \draw[decorate, decoration={brace, amplitude=5pt, mirror}, line width=0.6pt, black!50] 
      ($(P0)+(0.05,-0.2)$) -- ($(P1)+(0.05,-0.2)$) 
      node[midway, below=6pt, font=\scriptsize\itshape, black!60] {$d$};

    \draw[decorate, decoration={brace, amplitude=5pt, mirror}, line width=0.6pt, black!50] 
      ($(Q1)+(0,0.2)$) -- ($(Q0)+(0,0.2)$) 
      node[midway, above=6pt, font=\scriptsize\itshape, black!60] {$d' \geq d$};
  \end{tikzpicture}
  \caption{The sampled configuration along a straight line is an expansion of the sampled configuration along a non-linear path.}
\end{figure}
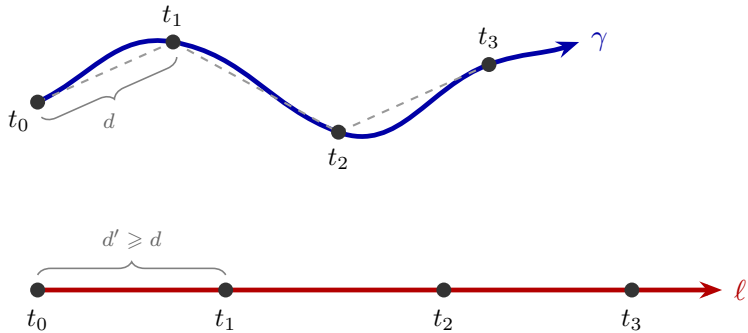

Since Theorem~\ref{thm:KP2} is formulated for discrete sets, our argument proceeds by sampling the path at finitely many times and applying the theorem to the resulting point configurations. Passing to the limit allows us to apply inequality \eqref{eqn:KP2} to the integral in \eqref{eqn:area_integral}, which swiftly yields the result by demonstrating that a straight line path minimises the intersection area and thus the expected escape time.

\begin{proof}[Proof of Theorem~\ref{thm:2d}]
  From \eqref{eqn:area_integral}, it suffices to prove that at any time $t>0$, the area of the non-escape region for an arbitrary path $\gamma$ is bounded below by that of a straight line path $\ell$. That is, we must show $\area(S_{\gamma}(t)) \ge \area(S_{\ell}(t))$ at any given time.

  Without loss of generality, due to rotational symmetry, let $\ell(s) = (s, 0)$. Fix $t > 0$ and let $P = \{ 0 = t_0 < t_1 < \dots < t_k = t \}$ be a partition of $[0, t]$. Consider $\vec{p} = (\vec{p}_0, \dots, \vec{p}_k)$ where $\vec{p}_i = -\gamma(t_{i})$, and $\vec{q} = (\vec{q}_0, \dots, \vec{q}_k)$ where $\vec{q}_i = -\ell(t_i)$, which are the centers of the translated balls along $-\gamma$ and $-\ell$, respectively. Define \[
    K_P(\gamma) \coloneq \bigcap_{i = 0}^{k} \ball(\vec{p}_i), \qquad 
    K_P(\ell) \coloneq \bigcap_{i = 0}^{k} \ball(\vec{q}_i)
  \] to be discrete approximations of the non-escape regions of $\gamma$ and $\ell$ respectively.

  Since the path $\gamma$ is arc-length parametrised, the Euclidean distance between two points on $\gamma$ satisfies \[
    \norm{\gamma(u) - \gamma(v)} \le \abs{u - v}.
    \] Hence for any pair of times $t_i, t_j \in P$, we have \[
    \norm{\gamma(t_i) - \gamma(t_j)} \le \abs{t_i - t_j} = \norm{\ell(t_i) - \ell(t_j)},
    \] with the last equality following from the definition of $\ell$. Thus, $\norm{\vec{p}_i - \vec{p}_j} \le \norm{ \vec{q}_i - \vec{q}_j}$, so $\vec{q}$ is an expansion of $\vec{p}$, and applying Theorem~\ref{thm:KP2} yields \[
    \area(K_{P}(\gamma)) \ge \area(K_{P}(\ell)).
    \] Moreover, we have \[
    S_{\ell}(t) = \bigcap_{s \in [0, t]} \ball(-\ell(s)) \subseteq K_P(\ell),
  \] which means 
  \begin{equation} \label{eqn:area_inequality}
    \area(K_{P}(\gamma)) \ge \area(S_{\ell}(t))
  \end{equation}
  for every finite partition $P$ at time $t > 0$.

  Let $(P_n)$ be a sequence of partitions of $[0, t]$ such that $P_{n+1}$ is a refinement of $P_n$ and $\bigcup P_n$ is dense in $[0, t]$. Set \[
    K_{n} \coloneq K_{P_n}(\gamma) = \bigcap_{s \in P_n} \ball(-\gamma(s)).
    \] Since $P_{n+1} \supseteq P_{n}$, more disks are being intersected, so $K_{n+1} \subseteq K_n$, forming a decreasing sequence of compact sets. Clearly, as $S_{\gamma}(t)$ intersects over all $s \in [0, t]$, we have $S_{\gamma}(t) \subseteq K_i$ for all $i$. By continuity of $\gamma$ and the density of the partitions, we obtain \[
    \bigcap_{n = 1}^{\infty} K_n = S_{\gamma}(t).
    \] Continuity of the Lebesgue measure from above gives us \[
    \lim_{n \to \infty} \area(K_n) = \area(S_{\gamma}(t)).
    \] Combining this with \eqref{eqn:area_inequality} yields the desired inequality \[
    \area(S_{\gamma}(t)) \ge \area(S_{\ell}(t)),
  \] and integrating over $t \in [0, \infty)$ implies $J(\gamma) \ge J(\ell)$.
\end{proof}

\section{Generalising to higher dimensions}

The argument used in the proof of Theorem~\ref{thm:2d} relied on the discrete Kneser--Poulsen theorem in $\R^2$ for arbitrary expansions. As such, the main thrust of that proof does not readily carry over to $\R^n$ without modification, primarily because the Kneser--Poulsen conjecture remains unproven in dimensions greater than $2$. It is worth noting that the conjecture has been resolved in the case of \emph{continuous} expansions in all dimensions, by Csik{\'o}s~\cite{Csikos1998}. A point configuration $\vec{q}$ is a \emph{continuous} expansion of $\vec{p}$ if there exists a motion $\vec{p}(t) = (\vec{p}_1(t), \dots, \vec{p}_N(t))$ such that $\vec{p}(0) = \vec{p}$ and $\vec{p}(1) = \vec{q}$, $\norm{\vec{p}_i(t) - \vec{p}_j(t)}$ is monotonically increasing, and $\vec{p}_i$ is continuous. However, an arbitrary discrete expansion in $\R^n$ need not arise from a continuous expansive motion in the same dimension, which is precisely the obstruction here.

If we were able to approximate a path by discrete segments (a \emph{polygonal chain}), and show that a straight line is a continuous expansion of the chain, then we would be able to invoke the continuous Kneser--Poulsen theorem to prove that a straight line path is optimal in $\R^n$. However, this is not always possible---Demaine and Eisenstat~\cite{Demaine2011} proved that there exist ``locked'' chains in $\R^n$ for $n \ge 3$ that cannot be straightened with an expansive motion. Despite this, they also showed that open chains in $\R^n$ \textit{can} always be straightened if the motion was allowed to move into a higher dimension. Specifically, they proved that any open chain in $\R^n$, for $n \ge 3$, can be straightened with an expansive motion when embedded into $\R^{n+1}$.

Crucially, Bezdek and Connelly~\cite{BezdekConnelly2002}, in the same paper in which they resolved the Kneser--Poulsen conjecture in the plane, proved an analogue of the conjecture which allows the motion to move into two extra dimensions.
\begin{theorem}[Bezdek--Connelly] \label{thm:bezdekconnelly}
  If \(\vec{p} = (\vec{p}_1, \dots, \vec{p}_N)\) and \(\vec{q} = (\vec{q}_1, \dots, \vec{q}_N)\) are two point configurations in \(\R^n\), such that \(\vec{q}\) is a piecewise-smooth expansion of \(\vec{p}\) in \(\R^{n+2}\), then
  \begin{equation*}
    \voln \biggl(\, \bigcup_{i = 1}^{N} \ball(\vec{p}_i, r_i) \biggr) \le \voln \biggl(\, \bigcup_{i = 1}^{N} \ball(\vec{q}_i, r_i) \biggr)
  \end{equation*} and 
  \begin{equation*}
    \voln \biggl(\, \bigcap_{i = 1}^{N} \ball(\vec{p}_i, r_i) \biggr) \ge \voln \biggl(\, \bigcap_{i = 1}^{N} \ball(\vec{q}_i, r_i) \biggr)
  \end{equation*}
  for all \(r_1 > 0, \dots, r_N > 0\), where \(\voln\) is the Lebesgue measure in $\R^n$.
\end{theorem}
Combining these facts allows us to prove a generalisation of Theorem~\ref{thm:2d} in largely the same way.

\begin{theorem}
  Let \(\gamma\) be any path in \(\R^n\). Then, the expected escape time satisfies \[
    J(\gamma) \ge J(\ell),
  \] where \(\ell\) is a straight line.
\end{theorem}
\begin{proof}
  As in the two-dimensional case, it suffices to prove that $\voln(S_{\gamma}(t)) \ge \voln(S_{\ell}(t))$ for any fixed $t > 0$.

  Let $P = \{0 = t_0 < t_1 < \dots < t_k = t\}$ be a partition of $[0, t]$, and define the point configuration $\vec{p} = (\vec{p}_0, \dots, \vec{p}_k)$ with $\vec{p}_i = -\gamma(t_i)$. This forms an \emph{open polygonal chain} in $\R^n$. Similarly, define $\vec{q} = (\vec{q}_0, \dots, \vec{q}_k)$, where $\vec{q}_i = -\ell(t_i)$, to be the corresponding point configuration for a straight line path. Without loss of generality due to rotational symmetry, we let $\ell(s) = (s, 0, \dots, 0)$, lying along the $x_1$-axis.

  Demaine and Eisenstat~\cite{Demaine2011} proved that any open chain in $\R^n$ can be straightened by an expansive motion when embedded in $\R^{n+1}$. Denote the expansion of $\vec{p}$ by $\vec{p}^*$. The specific motion given in their proof is analytic (as it is composed of sinusoidal functions), and thus $\vec{p}^*$ is a piecewise-smooth expansion of $\vec{p}$ in $\R^{n+1}$. We may view this motion as taking place in $\R^{n+2}$ by setting the $(n+2)$th coordinate to be zero, and thus the conditions of Theorem~\ref{thm:bezdekconnelly} are satisfied. Letting \[
    K_P(\gamma) = \bigcap_{i = 0}^{k} \ball(\vec{p}_i), \qquad 
    K_P(\ell) = \bigcap_{i = 0}^{k} \ball(\vec{q}_i)
  \] like before, we apply Theorem~\ref{thm:bezdekconnelly} to obtain
  \begin{equation} \label{eqn:inequality2}
    \voln(K_P(\gamma)) \ge \voln\biggl(\, \bigcap_{i = 0}^{k}\ball(\vec{p}^*_i) \biggr),
  \end{equation}
  where $\vec{p}^*_i$ is the $i$th point along the straightened configuration $\vec{p}^*$. The straightened chain $\vec{p}^*$ as constructed by the motion in \cite{Demaine2011} has edge lengths equal to the edge lengths of the original chain $\vec{p}$. That is, the expansion can be thought of as ``unfurling'' the chain, preserving the distances between consecutive vertices on the chain.

  To be precise, the distance between adjacent points $\vec{p}^*_i$ and $\vec{p}^*_{i+1}$ is $l_i = \norm{\vec{p}_{i+1} - \vec{p}_i}$. Since the volume of the intersection of balls is invariant under isometry, we may assume without loss of generality that $\vec{p}^*$ also lies on the $x_1$-axis, starting at the origin.

  In order to complete the proof, we must now construct a smooth expansion of $\vec{p}^*$ into $\vec{q}$. Let $L_j = \abs{t_{j+1} - t_j}$. Clearly, $l_j \le L_j$, due to the unit-speed parametrisation of the curve $\gamma$. Hence the desired expansion may be given by the motion $\vec{r}(\tau) = (\vec{r}_0(\tau), \dots, \vec{r}_k(\tau))$ where \[
    \vec{r}_j(\tau) = \Biggl(\,  \sum_{m = 0}^{j-1}\Bigl( \tau L_m + (1-\tau) l_m \Bigr), 0, \dots, 0 \Biggr),
    \] for $0 \le j \le k$ and $\tau \in [0, 1]$. This expansion is clearly smooth and keeps pairwise distances monotonically increasing, and we have $\vec{r}(0) = \vec{p}^*$, $\vec{r}(1) = \vec{q}$, so $\vec{q}$ is a (piecewise) smooth expansion of $\vec{p}^*$. Applying Theorem~\ref{thm:bezdekconnelly} a second time yields \[
    \voln\biggl(\, \bigcap_{i = 0}^{k}\ball(\vec{p}^*_i) \biggr) \ge 
    \voln\biggl(\, \bigcap_{i = 0}^{k}\ball(\vec{q}_i) \biggr) = \voln(K_P(\ell)),
    \] and by \eqref{eqn:inequality2}, we obtain \[
    \voln(K_P(\gamma)) \ge \voln(K_P(\ell)).
    \] Using the same continuity argument as in the proof of Theorem~\ref{thm:2d} implies \[
    \voln(S_{\gamma}(t)) \ge \voln(S_{\ell}(t)),
  \] and we obtain $J(\gamma) \ge J(\ell)$ by integrating over $t \in [0, \infty)$.
\end{proof}

\section{Expected linear distance to the boundary}

Having established the optimality of linear escape paths, we round out the investigation by computing the expected linear escape time. Since our paths are arc-length parametrised, this is equivalent to finding the expected linear distance to the boundary of the unit ball in $\R^n$. This quantity is closely related to classical chord-length questions for convex bodies. In particular, Kellerer~\cite{Kellerer1971} studies ``interior radiator randomness'' (a uniformly random interior point together with an independent uniformly random direction) and records, for the circle and sphere, the corresponding chord-length distribution. More recently, Hopkins et al.\ compute the average directional distance to the boundary of the unit disk and, in full generality, the unit $n$-ball~\cite[Theorem~2]{Hopkins2024}. Their answer is given in terms of a Wallis-type integral, and is equivalent to the closed form solution involving a ratio of gamma functions which we provide in Proposition~\ref{prop:linear-n}. We include our calculation here---it is shorter, geometric, and fits naturally with the symmetries already exploited in the preceding sections.

\subsection{Motivation: the disk}

We motivate our argument for the general case by starting in two dimensions in the unit disk $\ball^2 = \{(x, y) \in \R^2 : x^2 + y^2\le 1\}$. Choose a starting point $X$ uniformly in $\ball^2$, and independently, a random unit vector $U$ that represents the direction of travel. 
By rotational symmetry, we may fix $u = (1, 0)$, so the escape path is parallel to the $x$-axis. Let $R$ denote the $y$-coordinate of $X$, such that the path meets the $y$-axis at a point $Q(0, r)$, where $-1 \le r \le 1$. Let the chord intersect the circle at $A$ and $B$ as in Figure~\ref{fig:2d-diagram}.

\begin{figure}[htbp]
  \begin{tikzpicture}[
    scale=2.25,
    >=Stealth
    ]

    % Axes with arrows
    \draw[->,line width=0.6pt,black!60] (-1.3,0.) -- (1.3,0.) node[right,font=\small] {$x$};
    \draw[->,line width=0.6pt,black!60] (0.,-1.3) -- (0.,1.3) node[above,font=\small] {$y$};

    % Circle
    \draw[line width=0.6pt,blue!65!black] (0.,0.) circle (1.cm);

    \pgfmathsetmacro{\xp}{0.4}   % x-coordinate of P
    \pgfmathsetmacro{\rp}{0.45}    % y-coordinate of P

    % Chord endpoints at height y = rp
    \pgfmathsetmacro{\halfchord}{sqrt(1 - \rp*\rp)}
    \coordinate (A) at (\halfchord,\rp);
    \coordinate (B) at (-\halfchord,\rp);

    % P and its projection Q onto the y-axis
    \coordinate (P) at (\xp,\rp);
    \coordinate (Q) at (0,\rp);

    % Chord AB
    \draw[line width=0.8pt,color=red!70!black] (P)--(A);
    \draw[line width=0.8pt,color=black!40,dashed] (P)--(B);

    % Labels
    \node[anchor=west,font=\small, xshift = 2pt] at (A) {$A$};
    \node[anchor=east,font=\small, xshift = -2pt] at (B) {$B$};
    \node[anchor=north,font=\small, yshift = -2pt] at (P) {$P$};
    \node[anchor=north east,font=\small, yshift = -2pt] at (Q) {$Q(0,r)$};

    % Points
    \node[circle,fill=black!80,inner sep=1.25pt] at (P) {};
    \node[circle,fill=black!80,inner sep=1.25pt] at (A) {};
    \node[circle,fill=black!80,inner sep=1.25pt] at (B) {};
    \node[circle,fill=black!80,inner sep=1.25pt] at (Q) {};
  \end{tikzpicture}
  \caption{A linear escape path for the unit disk $\ball^2$.}
  \label{fig:2d-diagram}
\end{figure}
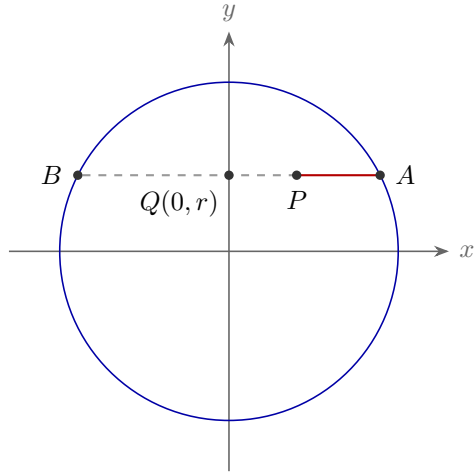

As we are just as likely to escape via $A$ as we are via $B$, it suffices to find the expected length of \[
  L(r) = \frac{1}{2} AP + \frac{1}{2} PB = \frac{1}{2} (AP + PB) = \frac{1}{2} AB = \sqrt{1 - r^2}.
  \] To calculate the marginal density of $R$, we fix $r$ and integrate over $x$ which has bounds $-\sqrt{1-r^2} \le x \le \sqrt{1-r^2}$. This gives $f_R(r) = \int_{-\sqrt{1-r^2}}^{\sqrt{1-r^2}} f(x, r) \diff x$, where $f(x, r)$ is the joint density. Since $X$ is uniform in the disk, we simply have $f(x, r) = 1/\area(\ball^2) = 1/\pi$. Hence $f_R(r) = 2\sqrt{1-r^2}/\pi$, and the expected escape time is thus \[
  \E[L] = \int_{-1}^{1} L(r)f_R(r) \diff r = \int_{-1}^{1} \frac{2}{\pi} \bigl(1 - r^2\bigr) \diff r = \frac{8}{3\pi}.
\] In other words, the expected linear distance to the boundary of the unit circle, starting from a uniformly distributed point within, is $\frac{8}{3\pi}$.

\subsection{Generalisation to higher dimensions}

Our calculation for the disk carries over to higher dimensions cleanly. The general idea remains the same---reduce to a single radial parameter $r$ and weight by the size of the corresponding level set. Denote the unit $n$-ball by $\ball^n \coloneq \{x \in \R^n : \norm{x} \le 1\}$, where $\norm{\blank}$ is the Euclidean norm in $\R^n$.

\begin{proposition}\label{prop:linear-n}
  Let \(X\) be a random variable distributed uniformly in \(\ball^n\). Let \(\ell\) be a straight line path starting at \(X\) with a direction chosen uniformly at random. Then, the expected escape time is \[
    J(\ell) = \frac{2}{\sqrt{\pi}}\, \frac{\Gamma\bigl(\frac{n+2}{2}\bigr)}{\Gamma\bigl(\frac{n+3}{2}\bigr)}.
  \]
\end{proposition}
\begin{proof}
  Due to rotational symmetry, we may assume without loss of generality that the direction of travel is $e_1 = (1, 0, \dots, 0)$, parallel to the $x_1$-axis. Write $x = (x_1, y)$ where $y \in \R^{n-1}$ represents the components orthogonal to $x_1$. 

  As we travel along the path, only the $x_1$-coordinate varies. Thus the path intersects the ball when $x_1^2 + \norm{y}^2 = 1$ which occurs when \[
    x_1 = \pm \sqrt{1 - \norm{y}^2}.
    \] Since the starting point is uniformly distributed along this chord, the expected distance to the boundary is simply half the length of the chord: \[
    L(y) = \sqrt{1 - \norm{y}^2}.
    \]We now calculate the marginal density of $Y$ by fixing $y$, the last $n-1$ coordinates, and integrating over $x_1$. If $y$ is held constant, the bounds for $x_1$ are $-\sqrt{1 - \norm{y}^2} \le x_1 \le  \sqrt{1 - \norm{y}^2}$, so \[
    f_Y(y) = \int_{-\sqrt{1 - \norm{y}^2}}^{\sqrt{1 - \norm{y}^2}} f(x_1, x_2, \dots, x_n) \diff x_1,
    \] where $f$ is the joint density. Since $X$ is uniform in $\ball^n$, the joint density is simply $\frac{1}{\omega_{n}}$, where $\omega_n$ denotes the volume of the unit ball $\ball^n$. Hence, we have\[
    f_Y(y) = \frac{2\sqrt{1-\norm{y}^2}}{\omega_n},
  \] and it remains to calculate
  \begin{equation} \label{eqn:ev-integral}
    \E[L] = \int_{\ball^{n-1}} \sqrt{1-\norm{y}^2}\, f_Y(y) \diff y = \frac{2}{\omega_n} \int_{\ball^{n-1}} 1 - \norm{y}^2 \diff y.
  \end{equation}
  Note that the integral is over $\ball^{n-1}$ because $\norm{y} \le 1$, implying that $y \in \ball^{n-1}$ since $y$ is a vector in $\R^{n-1}$.

  To evaluate this integral, we use spherical coordinates. Let $y = r\theta$, where $r = \norm{y} \in [0, 1]$ and $\theta \in \sphere^{n-2}$, where $\sphere^{n-2} = \partial \ball^{n-1}$ is the unit sphere in $\R^{n-1}$. Since the domain $\ball^{n-1}$ has dimension $n-1$, the volume element transforms as $\diff y = r^{n-2} \diff r \diff \sigma$, where $\sigma$ is the surface measure on $\sphere^{n-2}$. This gives \[
    \int_{\ball^{n-1}} 1 - \norm{y}^2 \diff y = \int_{\sphere^{n-2}} \int_{0}^{1} \bigl(1 - r^2\bigr)r^{n-2} \diff r \diff \sigma.
    \] Since the integrand is independent of $\sigma$, we can separate the integral to get \[
    \biggl(\, \int_{\sphere^{n-2}} \diff \sigma \biggr) \int_{0}^{1} \bigl(1 - r^2\bigr)r^{n-2} \diff r = \sigma_{n - 2} \int_{0}^{1} r^{n-2} - r^{n} \diff r,
  \] where $\sigma_{n-2}$ denotes the surface area of $\sphere^{n-2}$.

  The remaining integral evaluates to \[
    \int_{0}^{1} r^{n-2} - r^{n} \diff r = \frac{1}{n-1} - \frac{1}{n+1} = \frac{2}{n^2 - 1}.
    \] Substituting this back into \eqref{eqn:ev-integral} yields \[
    \E[L] = \frac{2}{\omega_n} \int_{\ball^{n-1}} 1 - \norm{y}^{2} \diff y = \frac{2}{\omega_n} \cdot \frac{2 \sigma_{n-2}}{n^2 - 1}.
    \] Using the standard identities \[
    \omega_n=\frac{\pi^{n/2}}{\Gamma\bigl(\frac{n}{2}+1\bigr)}, \qquad \sigma_{n-2}=\frac{2\pi^{(n-1)/2}}{\Gamma\bigl(\frac{n-1}{2}\bigr)}
    \] and simplifying gives \[
    \E[L] = \frac{2}{\sqrt{\pi}}\, \frac{\Gamma\bigl(\frac{n+2}{2}\bigr)}{\Gamma\bigl(\frac{n+3}{2}\bigr)}.
  \] Since $\ell$ is a unit-speed path, $J(\ell)$ is exactly this expected distance, completing the proof.
\end{proof}

\section{Open problems}

In this paper, we have established the optimality of a straight line escape path for forests shaped like a ball in $\R^n$. Our proof heavily relied on the Kneser--Poulsen conjecture to minimise the integral of an area function, which is the intersection of balls. Thus our argument does not generalise to other convex bodies.

Despite this, we suspect that a straight line path is still optimal in a large class of domains, especially regular polygons. Having performed numerical calculations that support this, we propose the following conjecture.

\begin{conjecture}
  Consider a forest in the shape of a regular polygon in $\R^2$ (e.g., an equilateral triangle or a square). We conjecture that among all unit-speed paths starting from a uniformly distributed point, the expected escape time is minimised by a straight line path.
\end{conjecture}

\bibliographystyle{amsplain}
\bibliography{references}

\end{document}